\documentclass[a4paper,10pt]{amsart}

\usepackage{graphicx}
\usepackage{dirtytalk}
\usepackage{amsthm}
\usepackage{amssymb}
\usepackage{amsmath}
\usepackage{accents}
\usepackage{hyperref}
\usepackage{xcolor}
\usepackage{enumerate}
\usepackage{listings}
\usepackage{complexity}
\usepackage{blkarray}
\usepackage{braket}
\usepackage{lmodern}
\usepackage[english]{babel}
\usepackage{enumitem}

\usepackage[font=small,labelfont=bf]{caption}
\usepackage[font=small,labelfont=normalfont,labelformat=simple]{subcaption}
\graphicspath{{figs/}}

\newtheorem{theorem}{Theorem}[section]
\newtheorem{definition}[theorem]{Definition}
\newtheorem{lemma}[theorem]{Lemma}

\newtheorem{corollary}[theorem]{Corollary}

\DeclareMathOperator{\stab}{stab}

\author
{
Rutger Campbell
}
\author 
{
J.~Pascal Gollin
}
\author
{
Kevin Hendrey
}
\author
{
Raphael Steiner 
}
\address[Campbell, Gollin, Hendrey]{Discrete Mathematics Group, Institute for Basic Science (IBS), 55 Expo-ro, Yuseong-gu, Daejeon, Korea, 34126}
\email{\tt $\{$rutger,pascalgollin,kevinhendrey$\}$@ibs.re.kr}
\address[Steiner]{Department of Computer Science, Institute of Theoretical Computer Science, ETH Z\"{u}rich, Switzerland}
\email{raphaelmario.steiner@inf.ethz.ch}

\thanks{The first and third authors were supported by the Institute for Basic Science (IBS-R029-C1).}
\thanks{The second author was supported by the Institute for Basic Science (IBS-R029-Y3).}
\thanks{The fourth author was supported by the SNSF Ambizione Grant No. 216071.}

\date{\today}

\title{Optimal bounds for zero-sum cycles. I. Odd order}

\begin{document}
\maketitle

\begin{abstract}
For a finite (not necessarily Abelian) group $(\Gamma,\cdot)$, let $n(\Gamma)$ denote the smallest positive integer $n$ such that for each labelling of the arcs of the complete digraph of order $n$ using elements from $\Gamma$, there exists a directed cycle such that the arc-labels along the cycle multiply to the identity. Alon and Krivelevich~\cite{AK} initiated the study of the parameter $n(\cdot)$ on cyclic groups and proved $n(\mathbb{Z}_q)=O(q \log q)$. This was later improved to a linear bound of ${n(\Gamma)\le 8|\Gamma|}$ for every finite Abelian group by M\'{e}sz\'{a}ros and the last author~\cite{MS}, and then further to $n(\Gamma)\le 2|\Gamma|-1$ for every non-trivial finite group independently by Berendsohn, Boyadzhiyska and Kozma~\cite{BBK} as well as by Akrami, Alon, Chaudhury, Garg, Mehlhorn and Mehta~\cite{AACGMM}. 

In this series of two papers we conclude this line of research by proving that $n(\Gamma)\le |\Gamma|+1$ for every finite group $(\Gamma,\cdot)$, which is the best possible such bound in terms of the group order and precisely determines the value of $n(\Gamma)$ for all cyclic groups as $n(\mathbb{Z}_q)=q+1$. 

In the present paper we prove the above result for all groups of odd order. The proof for groups of even order needs to overcome substantial additional obstacles and will be presented in the second part of this series.
\end{abstract}

\section{Introduction}\label{sec:intro}

In 1961, Erd\H{o}s, Ginzburg and Ziv~\cite{EGZ} proved that every sequence of $2m-1$ elements in the cyclic group $\mathbb{Z}_m$ contains a subsequence of length $m$ of total sum equal to zero.
This pioneering result sparked an active research area which has come to be known as ``zero-sum Ramsey theory''
 and which spans several disciplines including combinatorics, number theory, and algebra.
 For a general background on other foundational results and problems in this area, we refer the reader to the survey of Caro \cite{C}.

In this sequence of two papers we study a zero-sum Ramsey problem for group-labelled complete\footnote{A \emph{complete} digraph is a digraph having arcs $(x,y),(y,x)$ for any two distinct vertices $x,y$.} digraphs. 
For a group $(\Gamma,\cdot)$, a \emph{$\Gamma$-labelling} of a digraph $D$ is a map from the set $A(D)$ of arcs of $D$ to $\Gamma$, and the pair $(D,\gamma)$ is called a \emph{$\Gamma$-labelled digraph}.
We use the term \emph{balanced cycle} to refer to a directed cycle for which the cumulative product\footnote{We stick to the convention of using product notation and terminology for general groups, while reserving summation notation and terminology for Abelian groups. The topic originated with the study of Abelian groups, where balanced cycles are more commonly referred to as `zero-sum cycles'.} of its arc-labels in order along the cycle is equal to the identity (it is not hard to verify that this does not depend on which vertex of the cycle is taken as the starting point).
The main result of this paper is the following.
\begin{theorem}\label{thm:main}
    Let $\Gamma$ be a finite group of odd order.
    If $D$ is a complete digraph on~$|\Gamma|+1$ vertices,
    then for any $\Gamma$-labelling of $D$,
    there is a balanced cycle.
\end{theorem}
In general for a finite group $\Gamma$, we denote by $n(\Gamma)$ the smallest positive integer $n$ such that every $\Gamma$-labelled complete digraph on $n$ vertices contains a balanced cycle.
Theorem~\ref{thm:main} is best possible since the lower bound $n(\mathbb{Z}_q)\geq q+1$ holds for every cyclic group $\mathbb{Z}_q$.
To see this, for a complete digraph with an ordered vertex set, consider the $\mathbb{Z}_q$-labelling which assigns a generator $g$ to increasing arcs and the identity $\mathbf{0}$ to decreasing arcs, and note that every directed cycle contains between $1$ and $|V(D)|-1$ increasing arcs.  Thus, our result precisely determines the value of $n(\Gamma)$ for all cyclic groups of odd order, proving that $n(\mathbb{Z}_q)=q+1$ for all positive odd integers $q$. 
In fact, the above construction gives a complete $\mathbb{Z}_q$-labelled digraph on $q+1$ vertices in which there is a unique balanced cycle.
By deleting any arc of this cycle, we get a $\mathbb{Z}_q$-labelled digraph with no balanced cycle. 
Thus, the complete digraph on $q+1$ vertices is arc-critical
with respect to the property of admitting a $\mathbb{Z}_q$-labelling with no balanced cycles. 

The parameter $n(\Gamma)$ was first studied by Alon and Krivelevich~\cite{AK}, who proved that every graph containing a complete graph on $2n(\mathbb{Z}_q)$ vertices as a minor (i.e., of Hadwiger number at least $2n(\mathbb{Z}_q)$) necessarily contains a cycle of length divisible by $q$.
Using this result, we immediately obtain the following corollary of Theorem~\ref{thm:main}.
\begin{corollary}\label{cor:hadwigercycleramsey}
For every odd $q\in \mathbb{N}$, every graph containing $K_{2q+2}$ as a minor contains a cycle of length divisible by $q$.    
\end{corollary}
In the sequel to this paper, we will show that Theorem~\ref{thm:main} extends to groups of even order, and thus that Corollary~\ref{cor:hadwigercycleramsey} holds for all positive integers~$q$.
We choose to separate our results in this way since the proof for groups of odd order is significantly shorter and easier to digest, while containing many of the important ideas which are needed for the even case.
The increased difficulty of bounding~$n(\Gamma)$ for even ordered~$\Gamma$ arises from the existence of subgroups of order~$2$ and of order~$\frac{|\Gamma|}{2}$. 
Dealing with subgroups of these orders requires a long technical analysis which can be avoided entirely when considering groups of odd order.

We now present a brief summary of the literature regarding the parameter $n(\Gamma)$.
In their initial paper, Alon and Krivelevich proved the bound $n(\mathbb{Z}_q)\le O(q \log q)$ via a beautiful probabilistic argument~\cite{AK}. They also showed an improved bound of $n(\mathbb{Z}_p)\le 2p-1$ when $p$ is prime. 
Their result on cyclic groups was improved to the linear bound $n(\Gamma)\le 8|\Gamma|$ for all Abelian groups $(\Gamma,+)$ by M\'{e}szar\'{o}s and the last author~\cite{MS}, who also showed that $n(\mathbb{Z}_p)\le \frac{3p}{2}$ for every prime $p$ (which was the best known result for prime ordered groups prior to our paper). 
The previous best known bound for general groups of non-prime order was obtained independently by Berendsohn, Boyadzhiyska and Kozma~\cite{BBK} and by Akrami, Alon, Chaudhury, Garg, Mehlhorn and Meta~\cite{AACGMM}, who showed that $n(\Gamma)\le 2|\Gamma|-1$ for every non-trivial finite group $(\Gamma,\cdot)$. Somewhat orthogonally to this line of research, Letzter and Morrison~\cite{ML} recently studied the problem of bounding $n(\Gamma)$ for groups that are far from cyclic, and obtained the sublinear bound of $n(\mathbb{Z}_p^k)\le O(pk(\log k)^2)$ for large powers of cyclic groups of prime order.

\paragraph*{\textbf{Overview.}} 
The rest of this paper is organized as follows. In the next paragraph, we collect some notation, terminology and basic definitions related to digraphs, groups and group-labellings of digraphs. In Section~\ref{sec:key}, we present and prove our key lemma (Lemma~\ref{lemma:getnontrivialstabilizer}), which is central to the proof of Theorem~\ref{thm:main}.
This key lemma is already strong enough to give a direct proof of the equality $n(\mathbb{Z}_p)=p+1$ for all prime numbers $p$, and we present the short deduction of this special case in Section~\ref{sec:key} as well. After that, in Section~\ref{sec:proof}, we complete the proof of Theorem~\ref{thm:main}. 

\medskip

\paragraph*{\textbf{Notation and terminology.}}
In the following, let $D$ be a digraph, whose vertex and arc sets we denote by~$V(D)$ and~$A(D)$, respectively. As usual, for a subset $X \subseteq V(D)$ we denote by $D[X]$ the subdigraph of $D$ induced by $X$, and by~${D-X}$ the subdigraph $D[V(D)\setminus X]$ obtained by deleting $X$. Let $(\Gamma,\cdot)$ be a group, and let $\gamma$ be a $\Gamma$-labelling of $D$. 
For a directed path~$P$ in $D$ with vertex set ${\{ v_i \colon i \in [\ell+1] \}}$ and arc set~${\{ (v_{i}, v_{i+1}) \colon i \in [\ell] \}}$, we denote by~$\gamma(P)$ the cumulative product of the arc-labels in order along the path, that is ${\gamma(P) := \prod_{i=1}^{\ell} \gamma(v_i, v_{i+1})}$. 
Note that if $P$ is a path consisting only of a single vertex, then this is an empty product, and therefore $\gamma(P)$ equals the neutral element $\mathbf{1}$. 
For a subset $X$ of $V(D)$ and vertices $x, y \in X$ we denote by $\mathcal{P}_{D}(X,x)$ the set of all directed paths in $D[X]$ that end at $x$ (including the path of length $0$ that consists only of $x$ itself), and similarly, we denote by $\mathcal{P}_{D}(X,x,y)$ the set of all directed paths in $D[X]$ that start at $x$ and end at $y$. 
We denote by ${R_{D,\gamma}(X,x):=\{\gamma(P)| P \in \mathcal{P}_{D}(X,x)\}}$ and ${R_{D,\gamma}(X,x,y):=\{\gamma(P)| P \in \mathcal{P}_{D}(X,x,y)\}}$ the set of $\gamma$-values attained by these paths.
Given a subset~${S \subseteq \Gamma}$ of elements, we denote by 
\[
    \stab_l(S) := \{g \in \Gamma \colon g \cdot S = S \},\ 
    \stab_r(S) := \{g \in \Gamma \colon S \cdot g = S\}
\]
the \emph{left and right stabilizers} of the set~$S$ respectively. 
Note that these form subgroups of~$\Gamma$, and that $S$ can be written as a union of cosets of $\stab_l(S)$ or of cosets of $\stab_r(S)$.

Finally, let us introduce an important operation that we call \emph{shifting} that can be used to modify a given $\Gamma$-labelling without changing the set of balanced cycles. 
This operation also played a pivotal role in previous papers~\cite{AACGMM,BBK,ML,MS} on this topic.

We say that a $\Gamma$-labelling $\gamma'$ of $D$ is obtained from~$\gamma$ \emph{by shifting by $g\in \Gamma$ at $v\in V(D)$} if
\begin{enumerate}
    \item $\gamma'(u,v)=\gamma(u,v)\cdot g^{-1}$ for every in-neighbour $u$ of $v$,
    \item $\gamma'(v,w)=g\cdot\gamma(v,w)$ for every out-neighbour $w$ of $v$, and 
    \item $\gamma'(e)=\gamma(e)$ for every arc $e$ of $D-v$.
\end{enumerate}

If two $\Gamma$-labellings can be obtained form each other via a sequence of shifting operations, we say that they are \emph{shifting-equivalent}. Pause to note that shifting-equivalent labellings have the same collection of balanced cycles. In particular, if there are no balanced cycles in a $\Gamma$-labelled digraph $(D,\gamma)$, then there are also no balanced cycles with respect to any $\Gamma$-labelling that is shifting-equivalent to $\gamma$. 

Another operation on $\Gamma$-labellings that we need is called \emph{inverting}. Given a $\Gamma$-labelling $\gamma$ of a complete digraph $D$, we denote by $\mathrm{inv}(\gamma)$ the $\Gamma$-labelling of $D$ defined as $$\mathrm{inv}(\gamma)(u,v):=\gamma(v,u)^{-1}$$ for every pair of distinct $u, v \in V(D)$. Pause to note the following three facts about the inversion operation: (1) We have $\mathrm{inv}(\mathrm{inv}(\gamma))=\gamma$ for every $\Gamma$-labelling $\gamma$; (2)~$(D,\gamma)$ has no balanced cycles if and only if $(D,\mathrm{inv}(\gamma))$ has no balanced cycles; and (3)~for two $\Gamma$-labellings $\gamma$ and $\gamma'$ of $D$, we have that $\gamma$ and $\gamma'$ are shifting-equivalent if and only if  $\mathrm{inv}(\gamma)$ and $\mathrm{inv}(\gamma')$ are shifting-equivalent.
\section{Key Lemma and proof for groups of prime order}\label{sec:key}
In this section, we present our key lemma, and use it to quickly deduce that $n(\mathbb{Z}_p)=p+1$ for every prime $p$.

\begin{lemma}\label{lemma:getnontrivialstabilizer}
Let $(\Gamma,\cdot)$ be a non-trivial finite group. Let $(D,\gamma)$ be a complete $\Gamma$-labelled digraph on $|\Gamma|$ vertices which contains no balanced cycles. Then there exists a non-empty subset $X$ of $V(D)$ and a vertex $x \in X$ such that $|\mathcal{R}_{D,\gamma}(X,x)| \ge |X|$ and $\mathcal{R}_{D,\gamma}(X,x)$ has a non-trivial right stabilizer. 
\end{lemma}
\begin{proof}
Let us consider the set collection $$\mathcal{U}:=\{U\in 2^{V(D)}\setminus \{\emptyset\}|\forall u \in U:\text{ }|\mathcal{R}_{D,\gamma}(U,u)|<|U|\}.$$ 
Consider first the case that $\mathcal{U}=\emptyset$. In particular, we then have $V(D) \notin \mathcal{U}$. Thus there is a vertex $x \in V(D)$ such that $|\mathcal{R}_{D,\gamma}(V(D),x)| \ge |V(D)|=|\Gamma|$, which implies $\mathcal{R}_{D,\gamma}(V(D),x)=\Gamma$. Hence $\stab_r(\mathcal{R}_{D,\gamma}(V(D),x))=\stab_r(\Gamma)=\Gamma$ is indeed non-trivial, and the claim of the lemma is satisfied with $X:=V(D)$.

Thus we may assume that $\mathcal{U}\neq \emptyset$. Let us pick $U$ as an inclusion-wise minimal member of $\mathcal{U}$. Note that $U$ has at least two elements, since for every $u\in V(D)$ we have ${|\mathcal{R}_{D,\gamma}(\{u\},u)|=|\{\mathbf{1}\}|=1}$, so no singleton belongs to $\mathcal{U}$.

We now define $F$ to be the auxiliary digraph with vertex-set $V(F):=U$ such that an ordered pair $(u,v)$ of vertices is in $A(F)$ if and only if $|\mathcal{R}_{D,\gamma}(U\setminus \{u\},v)| \ge |U|-1$. Observe that every vertex in $F$ has out-degree at least one. Indeed, given a vertex $u \in U$ we have that $U\setminus\{u\} \notin \mathcal{U}$ by minimality, and thus there exists some vertex $v \in U\setminus\{u\}$ such that $|\mathcal{R}_{D,\gamma}(U\setminus \{u\},v)| \ge |U\setminus\{u\}|=|U|-1$, meaning $(u,v) \in A(F)$. 
Since $F$ is finite, it follows that there is a directed cycle $C=v_0v_1 \ldots v_{\ell-1}v_0$ in $F$. 
In the following, we will use index-addition modulo $\ell$. 
For every $i \in \{0,\ldots,\ell-1\}$, let us define $\mathcal{R}_i:=\mathcal{R}_{D,\gamma}(U\setminus \{v_{i-1}\},v_i)$, and note that $|\mathcal{R}_i| \ge |U|-1$ by the definition of $F$. 
Note also that for every $i\in \{\{0,\ldots,\ell-1\}$ we have 
$$ \mathcal{R}_i\cup(\mathcal{R}_{i+1}\cdot\gamma(v_{i+1},v_i)) \subseteq \mathcal{R}_{D,\gamma}(U,v_i),$$  by the definition of $\mathcal{R}_i$ and the fact that every directed path~$P$ in~${D[U\setminus\{v_i\}]}$ ending at~$v_{i+1}$ can be extended to a directed path in~$D[U]$ ending at~$v_i$ via the arc~${(v_{i+1},v_i)}$. 
Since $U \in \mathcal{U}$ we have that $\mathcal{R}_{D,\gamma}(U,v_i)$ is of size at most $|U|-1$. Hence, the above inclusion implies that $\mathcal{R}_i=\mathcal{R}_{i+1}\cdot \gamma(v_{i+1},v_i)$. 

Consequently, we have 
\begin{align*}
    \mathcal{R}_0&=\mathcal{R}_1\cdot\gamma(v_{1},v_0)=(\mathcal{R}_2\cdot \gamma(v_2,v_1))\cdot\gamma(v_1,v_0)=\dots \\
    &=\mathcal{R}_{\ell-1}\cdot\gamma(v_{\ell-1},v_{\ell-2})\cdot\phantom{.}\dots\phantom{.}\cdot\gamma(v_1,v_0)=\mathcal{R}_0\cdot\gamma(v_0,v_{\ell-1})\cdot \ldots \cdot\gamma(v_1,v_0).
\end{align*}
Finally, this means that $\gamma(v_0,v_{\ell-1})\cdot\gamma(v_{\ell-1},v_{\ell-2})\cdot\ldots\cdot\gamma(v_1,v_0) \in \stab_r(\mathcal{R}_0)$. Since by assumption the directed cycle $v_0 v_{\ell-1}\ldots v_1v_0$ in $D$ is not balanced, it follows that the right-stabilizer of $\mathcal{R}_0=\mathcal{R}_{D,\gamma}(U\setminus\{v_{\ell-1}\},v_0)$ is a non-trivial subgroup of $\Gamma$. The claim of the lemma is now satisfied with $X:=U\setminus \{v_{\ell-1}\}$ and~$x:=v_0 \in X$.
\end{proof}

Let us now show how one can use Lemma~\ref{lemma:getnontrivialstabilizer} to deduce the correctness of Theorem~\ref{thm:main} for groups of prime order. To do so, it will be convenient (also for the proof of Theorem~\ref{thm:main} in the next section) to introduce the notion of efficient and super-efficient tuples, as follows. 

\begin{definition}
Let $(\Gamma,\cdot)$ be a group, and $(D,\gamma)$ a $\Gamma$-labelled complete digraph. An \emph{\textcolor{red}{efficient}} tuple in $(D,\gamma)$ is of the form $(X,u,v,\gamma',R)$, where $X\subseteq V(D)$ and $u, v \in X$ are distinct vertices, $\gamma'$ is a $\Gamma$-labelling of $D$ that is shifting-equivalent to $\gamma$ and $R\subseteq \mathcal{R}_{D,\gamma'}(X,u,v)$ has a non-trivial right stabilizer and satisfies $|R|\ge |X|-1$.

If we even have $|R|\ge |X|$, then we call $(X,u,v,\gamma',R)$ \textcolor{red}{\emph{super-efficient}}.
\end{definition}

Let us note (for later use) the following consequence of our key Lemma~\ref{lemma:getnontrivialstabilizer}, which guarantees the existence of efficient tuples in balanced-cycle-free $\Gamma$-labelled complete digraphs. It immediately implies the statement of Theorem~\ref{thm:main} for groups of prime order, which we deduce thereafter.

\begin{lemma}\label{lem:saturatedexists}
Let $(\Gamma,\cdot)$ be a finite group. Let $(D,\gamma)$ be a complete $\Gamma$-labelled digraph on $|\Gamma|+1$ vertices which contains no balanced cycles. Then there exists an efficient tuple in $(D,\gamma)$.
\end{lemma}
\begin{proof}
Let $v_0$ be an arbitrary vertex of $D$. Let $\gamma'$ be obtained from $\gamma$ by shifting by $\gamma(v_0,w)$ at $w$ for every $w\in V(D)\setminus \{v_0\}$ in some order. Then $\gamma'$ is shifting-equivalent to $\gamma$ (so in particular, there are no balanced cycles in $D$ with respect to~$\gamma'$), and $\gamma'(v_0,w)=\mathbf{1}$ for every $w \in V(D)\setminus \{v_0\}$. 

By applying Lemma~\ref{lemma:getnontrivialstabilizer} to $D-v_0$ we obtain a subset $X'\subseteq V(D-v_0)$ and a vertex $x \in X'$ such that $|\mathcal{R}_{D,\gamma'}(X',x)|\ge |X'|$ and $\mathcal{R}_{D,\gamma'}(X',x)$ has a non-trivial right stabilizer. Let~$X:=\{v_0\} \cup X'$ and $R:=\mathcal{R}_{D,\gamma'}(X,v_0,x)$. We claim that~$(X,v_0,x,\gamma',R)$ is efficient.

To verify this, note that, since~$\gamma'(v_0,w)=\mathbf{1}$ for every~$w \in V(D)\setminus \{v_0\}$, we have ${R=\mathcal{R}_{D,\gamma'}(X,v_0,x)=\mathcal{R}_{D,\gamma'}(X',x)}$. Hence, we have ${|R|\ge |X'|=|X|-1}$ and~$R$ has a non-trivial right stabilizer, as desired.
\end{proof}

The reason why (super-)efficient tuples are useful in the context of finding balanced cycles in complete digraphs is that if $(X,u,v,\gamma',R)$ is an efficient tuple and the inverse value $\gamma'(v,u)^{-1}$ of the ``back-arc'' $(v,u)$ lies in $R$, then one can combine an appropriate $u$-$v$-path $P$ with value $\gamma'(P)=\gamma'(v,u)^{-1}$ together with the arc $(v,u)$ to find a balanced cycle.
To illustrate this, in the following we show how one can easily deduce the statement of Theorem~\ref{thm:main} for groups of prime order from Lemma~\ref{lem:saturatedexists}.

\begin{corollary}
For every prime number $p$, we have $n(\mathbb{Z}_p)=p+1$.
\end{corollary}
\begin{proof}
The lower bound $n(\mathbb{Z}_p)\ge p+1$ was already observed in~\cite{MS}, so let us prove $n(\mathbb{Z}_p)\le p+1$. Towards a contradiction, suppose that there exists a $\mathbb{Z}_p$-labelled complete digraph $(D,\gamma)$ on $p+1$ vertices with no balanced cycles. Then, by Lemma~\ref{lem:saturatedexists} there exists an efficient tuple $(X,u,v,\gamma',R)$ in $(D,\gamma)$. Then by definition, we have $|R|\ge |X|-1\ge 2-1=1$, so $R$ is non-empty, and $R$ has a non-trivial right stabilizer. Since the only non-trivial subgroup of $(\mathbb{Z}_p,+)$ is $\mathbb{Z}_p$ itself, this implies that $\stab_r(R)=\mathbb{Z}_p$ and hence $R=\mathbb{Z}_p$. Thus, we also have $-\gamma'(v,u)\in R\subseteq R_{D,\gamma'}(X,u,v)$. Let $P$ be a directed $u$-$v$-path in $D[X]$ such that $\gamma'(P)=-\gamma'(v,u)$. Then the directed cycle $P + (v,u)$ in $D$ is a balanced cycle in $(D,\gamma')$, and hence also a balanced cycle in $(D,\gamma)$, yielding the desired contradiction. This concludes the proof that $n(\mathbb{Z}_p)\le p+1$. 
\end{proof}
\section{Proof of Theorem~\ref{thm:main}}\label{sec:proof}
In this section, we present the full proof of our main result, Theorem~\ref{thm:main}. The rough outline is as follows: We proceed by contradiction, and consider a smallest counterexample to the theorem (in terms of the size of $|\Gamma|$) and a $\Gamma$-labelled digraph $(D,\gamma)$ on $|\Gamma|+1$ vertices without balanced cycles. The $\Gamma$-labelled digraph $(D,\mathrm{inv}(\gamma))$ then also forms a complete $\Gamma$-labelled digraph on $|\Gamma|+1$ vertices without balanced cycles.

The first step is to show that under these hypotheses, one can in fact find a super-efficient tuple in at least one of $(D,\gamma)$ and $(D,\mathrm{inv}(\gamma))$ (Claim~2). To do so, we show that a ``smallest'' among all efficient tuples in $(D,\gamma)$ or $(D,\mathrm{inv}(\gamma))$ (guaranteed to exist by Lemma~\ref{lem:saturatedexists}) can be augmented with two additional vertices to become super-efficient. This step relies on the fact that $\Gamma$ has no subgroups of order $2$ or $\frac{|\Gamma|}{2}$, which is the main reason that our result is so much easier to prove for odd groups than for even groups. Once Claim~2 is established, we proceed by considering a super-efficient tuple $(X,u,v,\gamma',R)$ where $X$ is inclusion-wise maximal. We then argue that one can either augment $X$ with two vertices to create a larger super-efficient tuple (which yields the desired contradiction), or that one can find a proper subgroup $\Gamma'$ of $\Gamma$ and a balanced-cycle-free $\Gamma'$-labelling of $D-X$, which we show to contain at least $|\Gamma'|+1$ vertices (this step requires the super-efficiency and would not work for efficient tuples). In this case, the desired contradiction then follows since we assumed $\Gamma$ to be a smallest counterexample. 

For both of the two main steps of the proof outlined above, it will be important to understand how the size of a subset $R$ of $\Gamma$ relates to the size $\{\mathbf{1},\mathbf{x}\}\cdot R$ for $\mathbf{x}\in \Gamma\setminus \{\mathbf{1}\}$, since this is the kind of transformation the subset $R$ of attainable path values in an efficient or super-efficient tuple will undergo when we add two new vertices to it. The following lemma gives a simple lower bound on the increase of the set-size in terms of the size of its right-stabilizer. This dependency on the size of the right stabilizer is the reason why in the definition of (super-)efficient tuples, we require the set $R$ to have a non-trivial right stabilizer: this property will guarantee a fast enough (``efficient'') increase of the size of the set $R$ of attainable path values when we augment the tuple with two carefully chosen vertices. 

\begin{lemma}\label{lemma:addtwoelements}
    Let~$S$ be a subset of a group~$\Gamma$.   If~${x \in \Gamma \setminus \stab_l(S)}$ and $S':=\{\mathbf{1},x\}\cdot S$, then $\stab_r(S)\subseteq\stab_r(S')$ and $|S'|\geq |S|+|\stab_r(S)|$. 
\end{lemma}
\begin{proof}
Let $h\in \stab_r(S)$.
By definition, $S\cdot h=S$, so 
\[S'\cdot h=(\{\mathbf{1},x\}\cdot S)\cdot h=\{\mathbf{1},x\}\cdot (S\cdot h)=\{\mathbf{1},x\}\cdot S=S',\] 
and so $h \in \stab_r(S')$.
Hence $\stab_r(S)\subseteq \stab_r(S')$.
Note that $S=\mathbf{1}\cdot S\subseteq S'$. Since $x\notin \stab_l(S)$, we know there is some element $y\in S'\setminus S$.
Let $C=y\cdot \stab_r(S)$. We claim that $C \subseteq S'\setminus S$. Indeed, we have 
\begin{align*}
C&=y\cdot \stab_r(S)\subseteq S'\cdot \stab_r(S)\\
&=(\{\mathbf{1},x\}\cdot S)\cdot \stab_r(S)=\{\mathbf{1},x\}\cdot (S\cdot \stab_r(S))=\{\mathbf{1},x\}\cdot S=S'
\end{align*} 
by definition of the stabilizer. Furthermore, suppose towards a contradiction that $C\cap S \neq \emptyset$, and let $z \in C \cap S$. Then, by definition of $C$, there exists some $s\in \stab_r(S)$ such that $z=y\cdot s$. Hence, $y=(y\cdot s)\cdot s^{-1}=z\cdot s^{-1}\in S \cdot s^{-1}=S$, since $s^{-1} \in \stab_r(S)$ (recall that the stabilizer forms a subgroup). This contradicts that $y \in S'\setminus S$, and so indeed we must have $C\subseteq S'\setminus S$. Finally, this implies $|S'|\ge |S|+|C|=|S|+|\stab_r(S)|$, as desired. 
\end{proof}

With all necessary tools at hand, we are now ready to present the proof of Theorem~\ref{thm:main}.

\begin{proof}[Proof of Theorem~\ref{thm:main}]
Suppose towards a contradiction that there exists a finite group $(\Gamma,\cdot)$ of odd order such that $n(\Gamma)> |\Gamma|+1$, and let $\Gamma$ be chosen as a smallest group with these properties. In particular, this assumption implies that $n(\Gamma')\le |\Gamma'|+1$ for every proper subgroup $\Gamma'$ of $\Gamma$ (since every subgroup of a group of odd order must also have odd order). 

In the following, let~$D$ denote a complete digraph on $|\Gamma|+1$ vertices. The assumption $n(\Gamma)>|\Gamma|+1$ by definition implies that there exists a $\Gamma$-labelling~$\gamma$ of~$D$ such that there are no balanced cycles in~$(D,\gamma)$.

By Lemma~\ref{lem:saturatedexists}, we find that there exists at least one efficient tuple in $(D,\gamma)$. In the following, let $(X^\ast,u^\ast,v^\ast,\gamma^\ast,R^\ast)$ be chosen among all efficient tuples in $(D,\gamma)$ or $(D,\mathrm{inv}(\gamma))$ such that $|X^\ast|$ is minimized. Replacing $\gamma$ by $\mathrm{inv}(\gamma)$ if necessary, we may assume w.l.o.g.\footnote{Note that this w.l.o.g. assumption is indeed legal: Since the labellings $\gamma$ and $\mathrm{inv}(\gamma)$ behave completely symmetrically, and in particular, since $\mathrm{inv}(\mathrm{inv}(\gamma))=\gamma$, we can replace $\gamma$ by $\mathrm{inv}(\gamma)$ if necessary, keeping all the necessary properties for $(D,\mathrm{inv}(\gamma)$ required to proceed with the rest of the proof in the same way as we do here for $(D,\gamma)$.} throughout the rest of the proof that $(X^\ast,u^\ast,v^\ast,\gamma^\ast,R^\ast)$ is in fact an efficient tuple in $(D,\gamma)$.  Our intermediate goal will be to show that we can augment the above tuple with two more vertices to become super-efficient. 
Towards this goal, we start by making the following simple but crucial observation.

\medskip

\paragraph*{\textbf{Claim 1.} Let $(X,u,v,\gamma',R)$ be any efficient tuple in $(D,\gamma)$. Then $R\neq \Gamma$ and $|V(D)\setminus X|\ge \max\{|\stab_l(R)|, |\stab_r(R)|\}$. If $(X,u,v,\gamma',R)$ is super-efficient, then we have $|V(D)\setminus X|\ge \max\{|\stab_l(R)|,|\stab_r(R)|\}+1$.}
\begin{proof}[Proof of Claim~1.]
First suppose towards a contradiction that $R=\Gamma$. This means $\mathcal{R}_{D,\gamma'}(X,u,v)=\Gamma$, since $R \subseteq \mathcal{R}_{D,\gamma'}(X,u,v)$ by the definition of an efficient tuple. 
In particular, we find that $\gamma'(v,u)^{-1} \in \mathcal{R}_{D,\gamma'}(X,u,v)$. Hence, there exists a directed path $P$ in $D$ from $u$ to $v$ such that $\gamma'(P)=\gamma'(v,u)^{-1}$. But now the directed cycle in $D$ formed by joining the arc $(v,u)$ to $P$ is balanced with respect to $\gamma'$, contradicting the fact that $\gamma'$ is shifting-equivalent to $\gamma$ and thus has no balanced cycles. Hence $R\neq \Gamma$.

For the second part of the claim, note that $R$ can be written as a disjoint union of right-cosets of the subgroup $\stab_l(R)$ of $\Gamma$, and also as a disjoint union of left-cosets of the subgroup $\stab_r(R)$. Since~$R\neq \Gamma$, this means that~$R$ must be disjoint from at least one right-coset of $\stab_l(R)$, and from at least one left-coset of $\stab_r(R)$. Hence, we have $|R|\le |\Gamma|-\max\{|\stab_l(R)|,|\stab_r(R)|\}$. Recalling that~${|R|\ge |X|-1}$, this implies
${|X|\le |\Gamma|-\max\{|\stab_l(R)|,|\stab_r(R)|\}+1}$.
If~$(X,u,v,\gamma',R)$ is super-efficient then~${|X| \le |\Gamma|-\max\{|\stab_l(R)|,|\stab_r(R)|\}}$ since~${|R|\ge |X|}$. The fact that~$|V(D)|=|\Gamma|+1$ now implies the second part of the claim.
\end{proof}

We next show that we can use $X^\ast$ and two additional vertices to build a super-efficient tuple in $(D,\gamma)$. 

\medskip

\paragraph*{\textbf{Claim 2.} There exists a super-efficient tuple in $(D,\gamma)$.}
\begin{proof}[Proof of Claim~2.]
Let $\delta$ be a $\Gamma$-labelling obtained from $\gamma^\ast$ by shifting, for every vertex $w \in V(D)\setminus X^\ast$, by value $\gamma^\ast(w,u^\ast)^{-1}$ at $w$ in some order. 
Then for every $w \in V(D)\setminus X^\ast$ we have $\delta(w,u^\ast)=\mathbf{1}$. Furthermore, $\gamma^\ast$ and $\delta$ agree on all arcs inside of $X^\ast$, which implies that~${\mathcal{R}_{D,\delta}(X^\ast,u^\ast,v^\ast)=\mathcal{R}_{D,\gamma^\ast}(X^\ast,u^\ast,v^\ast)}$. In particular,~${R^\ast \subseteq \mathcal{R}_{D,\delta}(X^\ast,u^\ast,v^\ast)}$.

Let $\Gamma^\ast:=\stab_l(R^\ast)$. By Claim~1, we have that $\Gamma^\ast$ is a proper subgroup of $\Gamma$ and that $|V(D)\setminus X^\ast|\ge |\Gamma^\ast|$. 

Suppose first that for every pair of distinct vertices $w_1,w_2 \in V(D)\setminus X^\ast$, we have $\delta(w_1,w_2)\in \Gamma^\ast$. Then the restriction of $\delta$ to the complete digraph $D-X^\ast$ forms a $\Gamma^\ast$-labelling, and thus the same is true for the restriction of $\mathrm{inv}(\delta)$. Recall that since $\Gamma^\ast$ is a proper subgroup of $\Gamma$, we must have $n(\Gamma^\ast)\le |\Gamma^\ast|+1$. Thus, if we were to have $|V(D)\setminus X^\ast|\ge |\Gamma^\ast|+1$, then there would need to be a $\delta$-balanced cycle within $D-X^\ast$. However, such a cycle does not exist, since $\delta$ is shifting equivalent to $\gamma$, which by assumption has no balanced cycles. Hence, we have $|V(D)\setminus X^\ast|\le |\Gamma^\ast|$, which together with the above implies $|V(D)\setminus X^\ast|=|\Gamma^\ast|$. Together with Claim~1, this implies that $|\Gamma^\ast|\ge |\stab_r(R^\ast)|$. Since $\stab_r(R^\ast)\neq \{\bf{1}\}$ (by definition of an efficient tuple), this also means that $\Gamma^\ast$ is non-trivial. 

The latter ensures that we may now apply Lemma~\ref{lemma:getnontrivialstabilizer} to the complete digraph $D-X^\ast$ of order $|\Gamma^\ast|$ equipped with the $\Gamma^\ast$-labelling obtained by restricting $\mathrm{inv}(\delta)$. This yields the existence of a subset $X'\subseteq V(D)\setminus X^\ast$ and a vertex $x \in X'$ such that $\mathcal{R}_{D,\mathrm{inv}(\delta)}(X',x)\ge |X'|$ and $\mathcal{R}_{D,\mathrm{inv}(\delta)}(X',x)$ has a non-trivial right stabilizer in $\Gamma^\ast$ (and hence also in $\Gamma$). Since $\delta(w,u^\ast)=\mathbf{1}$ for every $w \in V(D)\setminus X^\ast$,  we also have $\mathrm{inv}(\delta)(u^\ast,w)=\mathbf{1}$ for every $w \in V(D)\setminus X^\ast$. The latter implies that $\mathcal{R}_{D,\mathrm{inv}(\delta)}(\{u^\ast\} \cup X',u^\ast,x)=\mathcal{R}_{D,\mathrm{inv}(\delta)}(X',x)$. Since $\delta$ is shifting-equivalent to $\gamma$, we have that $\mathrm{inv}(\delta)$ is shifting-equivalent to $\mathrm{inv}(\gamma)$. All in all, this establishes that the tuple $(\{u^\ast\}\cup X',u^\ast,x,\mathrm{inv}(\delta),\mathcal{R}_{D,\mathrm{inv}(\delta)}(X',x))$ is efficient in $(D,\mathrm{inv}(\gamma))$. Using our minimality assumption on $X^\ast$, we obtain that $|X^\ast|\le |X'\cup \{u^\ast\}|=|X'|+1$. Adding $|V(D)\setminus X^\ast|=|\Gamma^\ast|$ to both sides of this inequality yields that
$$|\Gamma|+1=|V(D)|=|X^\ast|+|V(D)\setminus X^\ast|\le |X'|+1+|V(D)\setminus X^\ast|\le 2|\Gamma^\ast|+1.$$

However, since $\Gamma^\ast$ is a proper subgroup of the group $\Gamma$ of odd order, we must have $|\Gamma^\ast|\le \frac{|\Gamma|}{3}$. Plugging this into the above yields a contradiction, as desired.

Hence our assumption above must have been wrong: There exist two distinct vertices $w_1, w_2 \in V(D)\setminus X^\ast$ such that $\delta(w_1,w_2)\notin \Gamma^\ast$. 

Let us now consider the set $R':=\{\mathbf{1},\delta(w_1,w_2)\}\cdot R^\ast$. We claim that $R'$ is a subset of ${\mathcal{R}_{D,\delta}(X\cup \{w_1,w_2\},w_1,v^\ast)}$. Indeed, let $r \in R'$ be arbitrary. Then, since $R^\ast \subseteq \mathcal{R}_{D,\delta}(X^\ast,u^\ast,v^\ast)$, there exists a directed path $P$ in $D[X^\ast]$ from $u^\ast$ to $v^\ast$ such that $r \in \{\delta(P),\delta(w_1,w_2)\cdot \delta(P)\}$. Let $Q$ denote the $w_1$-$v^\ast$-path in $D[X^\ast \cup \{w_1,w_2\}]$ that is obtained by prepending to $P$ the arc $(w_1,u^\ast)$ if $r=\delta(P)$, or the length-two segment $(w_1,w_2),(w_2,u^\ast)$ if $r=\delta(w_1,w_2)\cdot\delta(P)$. We can see that in each case, $r=\delta(Q) \in \mathcal{R}_{D,\delta}(X\cup \{w_1,w_2\},w_1,v^\ast)$, as desired. This shows that indeed, $R'\subseteq \mathcal{R}_{D,\delta}(X\cup \{w_1,w_2\},w_1,v^\ast)$.

Finally, we claim that $(X^\ast \cup \{w_1,w_2\},w_1,v^\ast,\delta,R')$ is a super-efficient tuple, which will conclude the proof of the claim. All that remains to be verified for this is that $|R'|\ge |X^\ast \cup \{w_1,w_2\}|=|X^\ast|+2$ and that $R'$ has a non-trivial right stabilizer. To see this, note that $\delta(w_1,w_2)\notin \Gamma^\ast=\stab_l(R^\ast)$, and hence we can apply Lemma~\ref{lemma:addtwoelements}, which yields that $|R'|=|\{\mathbf{1},\delta(w_1,w_2)\}\cdot R^\ast|\ge |R^\ast|+|\stab_r(R^\ast)|$, and that $\stab_r(R')\supseteq \stab_r(R^\ast)$. The latter directly implies that $R'$ has a non-trivial right-stabilizer, and the former implies, using that $|\stab_r(R^\ast)|\ge 3$ (since $|\Gamma|$ is of odd order) that $|R'|\ge |R^\ast|+3\ge (|X^\ast|-1)+3=|X^\ast|+2$, as desired. 

This shows that indeed, $(X^\ast \cup \{w_1,w_2\},w_1,v^\ast,\delta,R')$ is a super-efficient tuple, concluding the proof of Claim~2.
\end{proof}
Having established the existence of a super-efficient tuple in $(D,\gamma)$, we can now quite easily conclude the proof. Let $(X,u,v,\gamma',R)$ be chosen among all super-efficient tuples in $(D,\gamma)$ such that $X$ is inclusion-wise maximal. Let $\Gamma':=\stab_l(R)$. Then by Claim~1, we have $R\neq \Gamma$ (and thus $\Gamma'$ is a proper subgroup of $\Gamma$) and $|V(D)\setminus X|\ge |\Gamma'|+1$. Let $\gamma''$ be the $\Gamma$-labelling of $D$ obtained from $\gamma'$ by, for every $w \in V(D)\setminus X$, shifting by value $\gamma'(w,u)^{-1}$ at $w$. Then $\gamma''(w,u)=\mathbf{1}$ for every $w \in V(D)\setminus X$, and $\gamma''$ is shifting-equivalent to $\gamma$.

Since $\Gamma'$ is a proper subgroup of $\Gamma$, we have $n(\Gamma')\le |\Gamma'|+1\le |V(D)\setminus X|$. 
Since there are no balanced cycles in $D-X$ with respect to $\gamma''$, this implies that there must be two distinct vertices $w_1, w_2 \in V(D)\setminus X$ such that ${\gamma''(w_1,w_2) \notin \Gamma'=\stab_l(R)}$. 
We can therefore apply Lemma~\ref{lemma:addtwoelements} to find that the set $R'':=\{\mathbf{1},\gamma''(w_1,w_2)\}\cdot R$ satisfies ${|R''|\ge |R|+|\stab_{r}(R)|}$ and $\stab_r(R'')\supseteq \stab_r(R)$. 
Since $\stab_r(R)$ is a non-trivial subgroup of $\Gamma$, this yields that $\stab_r(R'')$ is non-trivial and that ${|R''|\ge |R|+2\ge |X\cup \{w_1,w_2\}|}$. 
Pause to note that this implies that the tuple $(X\cup \{w_1,w_2\},w_1,v,\gamma'',R'')$ is super-efficient. This however contradicts our assumption that $X$ is inclusion-wise maximal among super-efficient tuples. This is the desired contradiction which concludes the proof of Theorem~\ref{thm:main}.
\end{proof}

\end{document}